\documentclass{amsart}
\usepackage{amsfonts}
\newtheorem{thm}{Theorem}[section]

\newtheorem{lem}[thm]{Lemma}

\def\ll{{ L}}
\begin{document}
\title[Leibniz algebras with associated Lie algebras $sl_2\dot{+} R$]
{Leibniz algebras with associated \\ Lie algebras $sl_2\dot{+} R$ ($dim R=2$).}
\author{L.M. Camacho, S. G\'{o}mez-Vidal, B.A. Omirov}
\address{[L.M. Camacho -- S. G\'{o}mez-Vidal] Dpto. Matem\'{a}tica Aplicada I.
Universidad de Sevilla. Avda. Reina Mercedes, s/n. 41012 Sevilla.
(Spain)} \email{lcamacho@us.es --- samuel.gomezvidal@gmail.com}
\address{[B.A. Omirov] Institute of Mathematics and Information Technologies
 of Academy of Uzbekistan, 29, F.Hodjaev srt., 100125, Tashkent (Uzbekistan)}
\email{omirovb@mail.ru}


\thanks{Partially supported by the PAICYT, FQM143 of the Junta de Andaluc\'{\i}a (Spain). The last named author was partially supported by IMU/CDC-program}

\begin{abstract}
From the theory of finite dimensional Lie algebras it is known that every finite
dimensional Lie algebra is decomposed into a semidirect sum of semisimple subalgebra and solvable radical. Moreover, due to work of Mal'cev the study of solvable Lie algebras is reduced to the study of nilpotent ones.

For the finite dimensional Leibniz algebras the analogues of the mentioned results are not proved yet. In order to get some idea how to establish the results we examine the Leibniz algebra for which the quotient algebra with respect to the ideal generated by squares elements of the algebra (denoted by $I$) is a semidirect sum of semisimple Lie algebra and the maximal solvable ideal. In this paper the class of complex Leibniz algebras, for which quotient algebras by the ideal $I$
are isomorphic to the semidirect sum of the algebra $sl_2$ and two-dimensional solvable ideal $R$, are described.

\end{abstract} \maketitle



\section{Introduction}

The notion of Leibniz algebra was introduced in 1993 by J.-L. Loday \cite{Lod1} as a generalization of Lie algebras. The last 20 years the theory of Leibniz algebras has been actively studied and many results of the theory of Lie algebras have been extended
to Leibniz algebras.

Due to Levi-Mal'cev's Theorem and results of Mal'cev in \cite{Mal1} the study of finite dimensional Lie algebras is reduced to nilpotent ones. From now on a lot of works are devoted to the description of finite-dimensional nilpotent Lie algebras see, for example,  \cite{Jac}, \cite{Kost}, \cite{Zel}.

In fact, the nilpotency of a finite-dimensional Lie algebra is characterized by Engel's
Theorem. In \cite{Kost} the local nilpotency of Lie algebra over a field of zero characteristic, satisfying the Engel's $n$-condition it is proved. Further, E.I. Zelmanov \cite{Zel} generalized this result to global nilpotency of a Lie algebra with Engel's $n$-condition. In \cite{Omi} the global nilpotency for the case of Leibniz algebra with Engel's $n$-condition was extended.

An algebra $L$ over a field $F$ is called
{\it Leibniz algebra} if for any elements $x, y, z \in L$
{\it the Leibniz identity} is holds true:
            $$
[x, [y, z]] = [[x, y], z] - [[x, z], y]
$$ where $[-,-]$ is multiplication of $L$.

Let $L$ be a Leibniz algebra and $I=ideal< [x,x]\  | \ x\in L >$ be the ideal of $L$ generated by all squares. Then $I$ is the minimal ideal with respect to the property that  $L/I$ is a Lie algebra. The natural epimorphism $\varphi  : L \rightarrow   L/I$ determines the corresponding Lie algebra $L/I$ of the Leibniz algebra $L.$

According to \cite{Jac} a 3-dimensional simple Lie algebra $\ll$ is said to be split if $\ll$ contains an element $h$ such that
 $ad(h)$ has a non-zero characteristic root $\rho$ belonging to the base field. Such algebra has a basis $\{e,f,h\}$ with the multiplication table
 $$\begin{array}{lll}
[e,h]=2e, & [f,h]=-2f, & [e,f]=h,\\{}
[h,e]=-2e, & [h,f]=2f, & [f,e]=-h.
\end{array}$$

This simple 3-dimensional Lie algebra is denoted by $sl_2$ and the basis $\{e,f,h\}$ is called {\it canonical basis}. Note that any 3-dimensional simple Lie algebra is isomorphic to $sl_2.$


The analogue of the Levi-Mal'cev's Theorem for Leibniz algebras is not proved yet. Moreover, for Leibniz algebra the situation when quotient algebra by ideal $I$ is isomorphic to a simple Lie algebra is not clarified, as well. In this direction we know only result of the paper  \cite{Rakh}, where the Leibniz algebras, for which quotient algebras by ideal $I$ are isomorphic to the simple Lie algebra $sl_2,$ are classified.

In fact, Dzhumadil'daev proposed the following construction of Leibniz algebras:

Let $G$ be a simple Lie algebra and $M$ be an irreducible skew-symmetric $G$-module (i.e. $[x,m]=0$ for all $x\in G, m\in M$). Then the vector space $Q=G+M$ equipped with the multiplication $$[x+m,y+n]=[x,y]+[m,y],$$ where $m, n\in M, x, y\in G$ is a Leibniz algebra.
Moreover, for this Leibniz algebra its corresponding Lie algebra is a simple algebra.

In fact, the notion of a simple Leibniz algebras was introduced in \cite{Abdykasimova}, \cite{AbdDzhum}.

Leibniz algebra $\ll $ is said to be {\it simple} if the only ideals of $\ll$ are $\{0\},$ $I,$ $\ll$ and $[\ll,\ll]\neq I.$ Obviously, in the case when a Leibniz algebra is Lie, the ideal $I$ is equal to zero. Therefore, this definition agrees with the definition of simple Lie algebra.

Note that the above mentioned Leibniz algebra is a simple algebra. It is also easy to see that the corresponding Lie algebra is simple for the simple Leibniz algebra.

In this paper, we study the class of complex Leibniz algebras, for which
its Lie algebra is isomorphic to the semidirect sum of the algebra $sl_2$ and a two-dimensional solvable ideal $R$.



 The representation of $sl_2$ is determined by the images $E,F,H$ of the base elements $e,f,h$ and we have
 $$[E,H]=2E, \ [F,H]=-2F, \ [E,F]=H, $$ $$[H,E]=-2E, \ [H,F]=2F, \ [F,E]=-H.$$

Conversely, any three linear transformations $E, F, H$ satisfying these relations determine a representation of $sl_2$ and hence a $sl_2$-module.

We suppose that a base field is the field of the complex numbers.
Then one has the following

\begin{thm}\label{jac} \cite{Jac}
For each integer $m=0,1,2,\dots$ there exists one and, in the sense of isomorphism, only one irreducible $sl_2$-module $M$ of dimension $m+1.$ The module $M$ has a basis $\{x_0,x_1,...,x_m\}$ such that the representing transformations $E,F$ and $H$ corresponding to the canonical basis $\{e,f,h\}$ are given by:

$$\begin{array}{ll}
H(x_k)=(m-2k)x_k, & 0\leq k\leq m, \\
F(x_m)=0, \ F(x_k)=x_{k+1}, & 0\leq k\leq m-1,  \\
E(x_0)=0, \ E(x_k)=-k(m+1-k)x_{k-1},& 1\leq k\leq m.\\
\end{array}$$
\end{thm}

In \cite{Rakh} using Theorem \ref{jac}, the authors described the complex finite dimensional Leibniz algebras whose $L/I$ is isomorphic to $sl_2.$

In this work, we consider the Leibniz algebra $\ll$ for which its corresponding Lie algebra is a semidirect sum of $sl_2$ and a two-dimensional solvable ideal $R.$ In addition we assume that $I$ is a right irreducible module over $sl_2.$

By verifying antisymmetric and Jacobi identities we derive that semidirect sum of $sl_2$ and a two-dimensional solvable Lie algebra is the direct sum of the algebras.

Let $\{x_0,x_1,\dots,x_m\}$ be a basis of $I$ and $\{e,f,h\}$ a basis of $sl_2.$ Thus, if $I$ is a right irreducible module over $sl_2$, then due Theorem \ref{jac} the products $[I,sl_2]$ are defined as follows:
$$\begin{array}{ll}
\, [x_k,h]=(m-2k)x_k & 0\leq k\leq m,\\
\, [x_k,f]=x_{k+1},  & 0\leq k\leq m-1, \\
\, [x_k,e]=-k(m+1-k)x_{k-1}, & 1\leq k\leq m, \\
 \end{array}$$
where omitted the products are equal to zero.


\section{Leibniz algebras with associated Lie algebras $sl_2\dot{+} R$}

Let $\ll$ be a Leibniz algebra such that $\ll/I\simeq sl_2\oplus R,$ where $R$ is a solvable Lie algebra and $\{\overline{e},\overline{h},\overline{f}\},$ $\{x_0,x_1,\dots,x_m\}$, $\{\overline{y_1},\overline{{y_2}},\dots,\overline{y_n}\}$ are the bases of $sl_2,$ $I$, $R$ respectively.

Let $\{e, h, f, x_0, x_1,\dots, x_m, y_1, {y_2},\dots, y_n\}$ be a basis
of the algebra $L$ such that $$\varphi(e)=\overline{e},\ \varphi(h)=\overline{h},\ \varphi(f)=\overline{f},\ \varphi(y_i)=\overline{y_i}, \ 1\leq i \leq n.$$

Then we have:
$$\begin{array}{lll}
\, [e,h]=2e+\sum\limits_{j=0}^m a_{eh}^jx_j, & [h,f]=2f+\sum\limits_{j=0}^m a_{hf}^jx_j, &[e,f]=h+\sum\limits_{j=0}^m a_{ef}^jx_j, \\
\, [h,e]=-2e+\sum\limits_{j=0}^m a_{he}^jx_j& [f,h]=-2f+\sum\limits_{j=0}^m a_{fh}^jx_j, & [f,e]=-h+\sum\limits_{j=0}^m a_{fe}^jx_j,\\
\, [e,y_i]= \sum\limits_{j=0}^m \alpha_{ij}x_j& [f,y_i]=\sum\limits_{j=0}^m \beta_{ij}x_j, & [h,y_i]=\sum\limits_{j=0}^m \gamma_{ij}x_j,\\
\, [x_k,h]=(m-2k)x_k, & 0\leq k\leq m,& \\
\, [x_k,f]=x_{k+1},  & 0\leq k\leq m-1, & \\
\, [x_k,e]=-k(m+1-k)x_{k-1}, & 1\leq k\leq m. &\\
 \end{array}$$
where $1\leq i\leq n.$

It is easy to check that similarly as in paper \cite{Rakh} one can get
$$\begin{array}{lll} \, [e,h]=2e, & [h,f]=2f, &[e,f]=h, \\
\, [h,e]=-2e& [f,h]=-2f, & [f,e]=-h,\\
\, [e,e]=0& [f,f]=0, & [h,h]=0.\\
 \end{array}$$

Let us denote the following vector spaces:
$$sl_2^{-1}=<e, h, f>, \ \ R^{-1}=<y_1, {y_2}, \dots y_n>.$$

The following result holds.
\begin{lem} \label{lema1} Let $L$ be a Leibniz algebra with condition
$L/I\cong sl_2\oplus R $, where $R$ is solvable ideal and $I$ is a right irreducible module over $sl_2$ with $m\neq3.$ Then $[sl_2^{-1}, R^{-1}]=0.$
\end{lem}

\begin{proof}

 It is known that it is sufficient to prove the equality for the basic elements of $sl_2^{-1}$  and $R^{-1}$.
Consider the Leibniz identity:
$$\begin{array}{ll}
[e,[e,y_i]] &=[[e,e],y_i] - [[e,y_i],e] = -[[e,y_i],e]=\\[2mm]
&=-\sum\limits_{j=0}^m \alpha_{ij}[x_j,e]=\sum\limits_{j=1}^m(-mj+j(j-1))\alpha_{ij}x_{j-1},\qquad 1\leq i\leq n.
\end{array}$$

On the other hand, we have that  $[e,[e,y_i]]
=[e,\sum\limits_{j=0}^m\alpha_{ij}x_j]=0$ for $1\leq i\leq n.$

Comparing the coefficients at the basic elements we obtain
$\alpha_{ij}=0$ for $1 \leq j \leq m,$ thus  $[e, y_i]=\alpha_{i,0}x_0$ with $1\leq i\leq n.$

Consider the chain of equalities
$$\begin{array}{ll}
0&=[e,\sum\limits_{j=0}^m\beta_{ij}x_j]=[e,[f,y_i]]=[[e,f],y_i]-[[e,y_i],f]=\\[2mm]
&=[h,y_i]-\alpha_{i,0}[x_0,f]=[h,y_i]-\alpha_{i,0}x_1.
\end{array}$$

Then we have that $[h, y_i] =\alpha_{i,0}x_1$ with $1\leq i\leq n$.

From the equalities
$$\begin{array}{ll}
0 &= [e,[h,y_i]] = [[e,h],y_i]-[[e,y_i],h] = 2[e,y_i]- \alpha_{i,0} [x_0,h]=\\[2mm]
&=2\alpha_{ i,0}x_0 -m \alpha_{i,0}x_0=\alpha_{ i,0}(2-m)x_0,
\end{array}$$
it follows that $\alpha_{i,0} =0$ for $1\leq i\leq n.$ Taking into account that $m\neq 2$ we get $[e, y_i]=[h,y_i]=0$ with $1\leq i\leq n.$

From the equalities
$$\begin{array}{ll}
0 &= [f,[e,y_i]] = [[f,e],y_i] - [[f,y_i],e] = [h,y_i] - [[f,y_i],e] = -[[f,y_i],e]=\\[2mm]
&=-\sum\limits_{j=0}^m \beta_{ij}[x_j,e]=-\sum\limits_{j=0}^m (-mj+j(j-1))\beta_{ij}x_{j-1},
\end{array}$$
we derive $\beta_{i,j} = 0$ for $1\leq j \leq m$. Consequence,
$[f,y_i]= \beta_{i,0}x_0,$ for all $1\leq i\leq n$.

Similarly, from
$$\begin{array}{ll}
0&=[f,[f,y_i]]=[[f,f],y_i]-[[f,y_i],f] = [[f,y_i],f]=\\[2mm]
&= \beta_{i,0}[x_0, f]=\beta_{i,0}x_1,
\end{array}$$
we obtain $[f,y_i]=0$ for all $1\leq i\leq n.$

Thus, we obtain $[e,y_i] = [f,y_i] =
[h,y_i]=0$ with $1\leq i\leq n,$ i.e $[sl_2^{-1},R^{-1}]=0.$
\end{proof}

\section{Leibniz algebras with associated Lie algebras $sl_2 \oplus R$ ($dim R=2$).}

Let $R$ be a two-dimensional solvable Lie algebra, then from the classification of two-dimensional Lie algebras (see \cite{Jac}) we know that in $R$ there exists a basis $\{\overline{y_1},\overline{{y_2}}\}$ with the following table of multiplication
$$[\overline{y_1}, \overline{{y_2}}] = \overline{y_1},\qquad  [\overline{{y_2}}, \overline{y_1}] = - \overline{y_1}.$$

In the case when $m\neq3$ and $I$ a right irreducible module over $sl_2$, summarizing the results of Lemma \ref{lema1} we get the following table of multiplication:
$$\begin{array}{lll}
\, [e,h]=2e, & [h,f]=2f, &[e,f]=h, \\
\, [h,e]=-2e& [f,h]=-2f, & [f,e]=-h,\\
\, [x_k,h]=(m-2k)x_k & 0\leq k\leq m,&\\
\, [x_k,f]=x_{k+1},  & 0\leq k\leq m-1, & \\
\, [x_k,e]=-k(m+1-k)x_{k-1}, & 1\leq k\leq m, &\\
\, [y_i,e]=\sum\limits_{j=0}^m a_{i e}^jx_j ,& 1\leq i\leq 2,&\\
\, [y_i,f]=\sum\limits_{j=0}^m a_{i f}^jx_j,&1\leq i\leq 2,& \\
\, [y_i,h]=\sum\limits_{j=0}^m a_{i h}^jx_j,& 1\leq i\leq 2,& \\
\, [x_k,y_i]=\sum\limits_{j=0}^m a_{ij}^k x_j ,& 0\leq k\leq m,&1\leq i\leq 2,  \\
\, [y_1,{y_2}]=y_1+\sum\limits_{j=0}^m a_{12}^jx_j ,& [{y_2},y_1]=-y_1,&\\
\, [y_1,y_1]=\sum\limits_{j=0}^m a_{1}^jx_j, & [{y_2},{y_2}]=\sum\limits_{j=0}^m a_{{2}}^jx_j,&
 \end{array}\eqno(1)$$
where $\{e, h, f, x_0, x_1, …, x_m, y_1, {y_2}\}$ is a basis of $L.$

Let us present the following theorem which describes the Leibniz algebras with condition
$L/I\cong sl_2\oplus R$, where $ m\neq 3,\
n=2$ and $I$ a right irreducible module over $sl_2$.

\begin{thm}
 Let $L$ be a Leibniz algebra with the condition
$L/I\cong sl_2\oplus R$, where $R$ is two-dimensional solvable ideal and  $I$ a right irreducible module over
$sl_2\ (m\neq3)$. Then there exists a basis $\{e, h, f,
x_0, x_1, …, x_m, y_1, {y_2}\}$ of the algebra $L$ such that the table of multiplication in $L$ have the following form:
$$\begin{array}{lll} \, [e,h]=2e, & [h,f]=2f, &[e,f]=h, \\
\, [h,e]=-2e& [f,h]=-2f, & [f,e]=-h,\\
\, [x_k,h]=(m-2k)x_k & 0\leq k\leq m,&\\
\, [x_k,f]=x_{k+1},  & 0\leq k\leq m-1, & \\
\, [x_k,e]=-k(m+1-k)x_{k-1}, & 1\leq k\leq m, &\\
\, [y_1,{y_2}]=y_1& [{y_2},y_1]=-y_1, &\\
\, [x_k,{y_2}]=a x_k,& 0\leq k\leq m, \quad a\in F \\
 \end{array}$$
where omitted products are equal to zero.
\end{thm}

\begin{proof}
 Let $L$ be an algebra satisfying the conditions of the theorem, then we get the table of multiplication $(1)$. Further we shall study the product $[I, R^{-1}]$.

We consider the chain of equalities
$$\begin{array}{ll}
0& = [x_i,[h,y_1]]=[[x_i,h],y_1]-[[x_i,y_1],h] = (m-2i)[x_i,y_1]-\sum\limits_{k=0}^m a_{1 k}^i[x_k,h]=\\[2mm]
&=(m-2i)\sum\limits_{k=0}^m a_{1 k}^ix_k-\sum\limits_{k=0}^m a_{1 k}^i(m-2k)x_k=\sum\limits_{k=0}^m a_{1 k}^i(m-2i-(m-2k))x_k=\\[2mm]
&=\sum\limits_{k=0}^m 2a_{1 k}^i(k-i)x_k,
\end{array}$$
from which we have $a_{1 k}^i= 0,$ with $0\leq i\leq m$ and $i\neq k.$ Thus, $[x_i,y_1]=a_{1 i}^i x_i=a_{1 i}x_i$ with $0\leq i\leq m.$

Similarly,
$$\begin{array}{ll}
0& = [x_i,[h,{y_2}]]=[[x_i,h],{y_2}]-[[x_i,{y_2}],h] = (m-2i)[x_i,{y_2}]-\sum\limits_{k=0}^m a_{{2} k}^i[x_i,h]=\\[2mm]
&=(m-2i)\sum\limits_{k=0}^m a_{{2} k}^ix_i-\sum\limits_{k=0}^m a_{{2} k}^i(m-2k)x_k=\sum\limits_{k=0}^m a_{{2} k}^i(m-2i-(m-2k))x_k=\\[2mm]
&=\sum\limits_{k=0}^m 2a_{{2} k}^i(k-i)x_k,
\end{array}$$
we get $[x_i,{y_2}]=a_{{2} i}^ix_i=a_{{2} i}x_i$ with $0\leq i\leq m.$

From the identity $[x_i,[y_1,{y_2}]]=[[x_i, y_1],{y_2}]-[[x_i, {y_2}],y_1],$
we deduce
$$\begin{array}{l}
[x_i,y_1+\sum\limits_{k=0}^m a_{1 {2}}^kx_k]=a_{1 i}[x_i, {y_2}] - a_{{2}i}[x_i, y_1]\quad \Rightarrow \\[2mm]
\Rightarrow \quad  [x_i, y_1] = a_{1i}a_{{2}i} x_i - a_{{2}i}a_{1i}x_i=0,
\end{array}$$
from which we have $[x_i, y_1]=0$ with $0\leq i\leq m,$ i.e. $[I,y_1]=0.$

We consider the identity $[x_i,[{y_2},e]] = [[x_i,{y_2}],e] - [[x_i, e],{y_2}]$ for $ 0 \leq i \leq m.$

Then
$$\begin{array}{ll}
0 &= a_{{2}i}[x_i,e] -
(-mi+i(i-1))[x_{i-1},{y_2}] =\\[2mm]
&= a_{{2}i}(-mi+i(i-1))x_{i-1}
-a_{{2},i-1}(-mi+i(i-1))x_{i-1}=\\[2mm]
&= - (-mi+i(i-1))( a_{{2}i}-a_{{2},i-1})x_{i-1} = 0,
\end{array}$$
which leads $a_{{2}i}=a_{{2},i-1}=a,$ i.e. $[x_i, {y_2}]= a x_i$ with $ 0\leq i \leq m$.

\

Now we shall study the products $[R^{-1}, R^{-1}]$ and $[R^{-1}, sl_2^{-1}]$.

\

Verifying the following
$$\begin{array}{ll}
0&=[y_1,\sum\limits_{j=0}^m a_{1f}^jx_j]=[y_1,[y_1,f]]=[[y_1,y_1],f] - [[y_1,f],y_1]=[[y_1,y_1],f]=\\[2mm]
&=\sum\limits_{j=0}^m a_{{1}}^j[x_j,f]=\sum\limits_{j=0}^{m-1} a_{{1}}^jx_{j+1},
\end{array}$$
we obtain $a_{1}^j=0$ for $0\leq j\leq m-1$, i.e. $[{y_1},{y_1}]= {a_{{1}}}^m x_m.$

Consider the equalities
$$0=[{y_1},[{y_1},h]]=[[{y_1},{y_1}],h]-[[{y_1},h],{y_1}]=[[{y_1},{y_1}],h]=a_{1}^m [x_m, h]=-m a_{1}^mx_m,$$
which deduce $a_{1}^m=0$, hence  $[{y_1},{y_1}]=0.$

From the following identities
$$0=[{y_2},[{y_1},h]]=[[{y_2},{y_1}],h]-[[{y_2},h],{y_1}] =
-[{y_1},h],$$ $$0=[{y_2},[{y_1},f]]=[[{y_2},{y_1}],f]-[[{y_2},f],{y_1}] =-[{y_1},f],$$
$$0=[{y_2},[{y_1},e]]=[[{y_2},{y_1}],e]-[[{y_2},e],{y_1}] = - [{y_1},e],$$
we obtain $[{y_1},h]=[{y_1},f]=[{y_1},e] = 0.$

Using the above obtained equalities and the following
$$\begin{array}{ll}
0& = [{y_1},[{y_2},f]] = [[{y_1},{y_2}],f] - [[{y_1},f],{y_2}] = [[{y_1},{y_2}],f]=\\[2mm]
&=[{y_1}+\sum\limits_{k=0}^m a_{{1}{2}}^kx_k,f]=\sum\limits_{k=0}^m
a_{{1}{2}}^k[x_k,f]=\sum\limits_{k=0}^{m-1} a_{{1}{2}}^kx_{k+1},
\end{array}$$
we get $a_{{1}{2}}^i=0$ with $0\leq i\leq m-1.$

Now from
$$\begin{array}{ll}
0& = [{y_1},[{y_2},h]] =[[{y_1},{y_2}],h] - [[{y_1},h],{y_2}]
=[[{y_1},{y_2}],h]=\\[2mm]
&=[{y_1}+a_{{1}{2}}^m x_m , h]=-ma_{{1}{2}}^mx_m,
\end{array}$$
we get $a_{{1}{2}}^m=0$, consequently  $a_{{1}{2}}^m=0$  for all $0 \leq i
\leq m$, i.e. $[{y_1}, {y_2}] = {y_1}.$

Thus, we obtain the following table of multiplication:
$$\begin{array}{lll}
\, [e,h]=2e, & [h,f]=2f, &[e,f]=h, \\
\, [h,e]=-2e& [f,h]=-2f, & [f,e]=-h,\\
\, [x_k,h]=(m-2k)x_k & 0\leq k\leq m,&\\
\, [x_k,f]=x_{k+1},  & 0\leq k\leq m-1, & \\
\, [x_k,e]=-k(m+1-k)x_{k-1}, & 1\leq k\leq m, &\\
\, [{y_2},e]=\sum\limits_{j=0}^m a_{{2}e}^jx_j ,&[{y_2},f]=\sum\limits_{j=0}^m a_{{2}f}^jx_j ,& [{y_2},h]=\sum\limits_{j=0}^m a_{{2}h}^jx_j, \\
\, [{y_1},{y_2}]={y_1},& [{y_2},{y_1}]=-{y_1},& [{y_2},{y_2}]=\sum\limits_{j=0}^m a_{{2}}^jx_j, \\
\, [x_k,{y_2}]=ax_k, &0\leq k\leq m. \\
 \end{array}$$

In order to complete the proof of the theorem we need to prove that $[{y_2},{y_2}] = 0,$ and $[R^{-1}, sl_2^{-1}] =0$.

Consider two cases:\\

\textbf{Case 1.}

Let $a\neq0$, then taking the change of the basic element as follows $${y_2}'={y_2}-\sum\limits_{j=0}^m\frac{a_{2}^{j}}{a}x_j,$$ we get
$$\begin{array}{ll}
[{y_2}',{y_2}']&=\left [{y_2}-\sum\limits_{j=0}^m\frac{a_{2}^j}{a}x_j,{y_2}-
\sum\limits_{j=0}^m\frac{a_{2}^j}{a}x_j\right ]=\\[2mm]
&=[{y_2},{y_2}]-\left [\sum\limits_{j=0}^m\frac{a_{2}^j}{a}x_j,{y_2}\right ]=
\sum\limits_{j=0}^m a_{2}^j x_j-\sum\limits_{j=0}^m a_{2}^j x_j=0,
\end{array}$$
which leads to $[{y_2},{y_2}] = 0.$

Consider
$$\begin{array}{ll}
0 &= [{y_2},[{y_2},h]] = [[{y_2},{y_2}],h] - [[{y_2},h],{y_2}] = -[[{y_2},h],{y_2}] =\\[2mm]
& = -\sum\limits_{j=0}^m a_{{2}h}^j[x_j,{y_2}]=-\sum\limits_{j=0}^m a_{{2}h}^jax_j,
\end{array}$$
which gives $a_{{2}h}^j=0$ for  $0 \leq   j \leq   m$.

Similarly from the equalities
$$\begin{array}{ll}
0 &= [{y_2},[{y_2},f]] = [[{y_2},{y_2}],f] - [[{y_2},f],{y_2}] =  -\sum\limits_{j=0}^ma_{{2}f}^j[x_j,{y_2}]=-\sum\limits_{j=0}^ma_{{2}f}^j a x_j,\\[2mm]
0 &= [{y_2},[{y_2},e]] = [[{y_2},{y_2}],e] - [[{y_2},e],{y_2}] =  -\sum\limits_{j=0}^ma_{{2}e}^j[x_j,{y_2}]=-\sum\limits_{j=0}^ma_{{2}e}^jax_j,
\end{array}$$
we get $a_{{2}f}^j=a_{{2}e}^j=0$ for $0\leq   j \leq m$.
Hence, $[R^{-1},sl_2^{-1}]=0$.

Thus, we proved the theorem for $a\neq0$.\\

\textbf{Case 2.}

Let $a=0$, then we consider the identity
$$[{y_2},[{y_2},f]] = [[{y_2},{y_2}],f] - [[{y_2},f],{y_2}]$$
and we derive
$$\begin{array}{l}
0=\sum\limits_{j=0}^ma_{2}^i[x_i,f]=\sum\limits_{j=0}^{m-1}a_{2}^ix_{i+1}\ \Rightarrow\ a_{2}^i=0, \ 0\leq i\leq m-1, \ \ i.e. \ \ [{y_2},{y_2}]=a_{2}^m x_m.
\end{array}$$

From the chain of the equalities $$0 = [{y_2},[{y_2},h]] = [[{y_2},{y_2}],h] -
[[{y_2},h],{y_2}]=a_{2}^m[x_m,h]=-m a_{2}^m x_m,$$ we obtain $a_{2}^m=0$, that is $[{y_2}, {y_2}] = 0.$

Let us take the change of the basic element in the form:
$${y_2}'={y_2}-\sum\limits_{j=1}^m\frac{a_{{2}e}^{j-1}}{-mj+j(j-1)}x_j$$

Then
$$\begin{array}{ll}
[{y_2}',e]&=[{y_2},e]-\sum\limits_{j=1}^m\frac{a_{{2}e}^{j-1}}{-mj+j(j-1)}[x_j,e]=\\[2mm]
&=[{y_2},e]-\sum\limits_{j=1}^m\frac{a_{{2}e}^{j-1}}{-mj+j(j-1)}(-mj+j(j-1))x_{j-1}=\\[2mm]
&=\sum\limits_{j=0}^m a_{{2}e}^j x_j-\sum\limits_{j=1}^{m} a_{{2}e}^{j-1}x_{j-1}=\\[2mm]
&=\sum\limits_{j=0}^m a_{{2}e}^j x_j-\sum\limits_{j=0}^{m-1} a_{{2}e}^j x_j=a_{{2}e}^m x_m.
\end{array}$$

Thus, we can assume that $$[{y_2},e]=a_{{2}e}^m x_m,\ \ \
[{y_2},h]=\sum\limits_{j=0}^ma_{{2}h}^jx_j,\ \ \
[{y_2},f]=\sum\limits_{j=0}^ma_{{2}f}^jx_j.$$

We have
$$\begin{array}{ll}
[{y_2},[e,h]]&=[[{y_2},e],h]-[[{y_2},h],e]=a_{{2}e}^m [x_m,h]-\sum\limits_{j=0}^ma_{{2}h}^j[x_j,e]=\\[2mm]
&=-ma_{{2}e}^mx_m-\sum\limits_{j=0}^ma_{{2}h}^j(-mj+j(j-1))x_{j-1}.
\end{array}$$

On the other hand $[{y_2},[e,h]] = 2[{y_2},e] =2a_{{2}e}^m x_m$.

Comparing the coefficients at the basic elements, we get $a_{{2}e}^m=0$ and
$a_{{2}h}^j=0$ where $1\leq j\leq m.$ Hence, $[{y_2},e]=0$,\ $[{y_2},f]=\sum\limits_{j=0}^m a_{{2}f}^jx_j$,\
$[{y_2},h]=a_{{2}h}^0 x_0.$

Consider
$$\begin{array}{ll}
[{y_2},[e,f]]&=[[{y_2},e],f]-[[{y_2},f],e]=-\sum\limits_{j=0}^m a_{{2}f}^j[x_j,e]=\\
&=-\sum\limits_{j=0}^ma_{{2}f}^j(mj+j(j-1))x_{j-1}=\\[2mm]
&=ma_{{2}f}^1x_0-\sum\limits_{j=2}^ma_{{2}f}^j(mj+j(j-1))x_{j-1}.
\end{array}$$

On the other hand
$$[{y_2},[e,f]] = [{y_2},h]=a_{{2}h}^0x_0.$$

Comparing the coefficients, we obtain $a_{{2}h}^0=m a_{{2}f}^1$
and $a_{{2}f}^j=0$ for $2\leq j\leq m.$
Then we have the product $[{y_2}, f] = a_{{2}f}^0 x_0+a_{{2}f}^1 x_1.$

Now we consider the equalities
$$-2[{y_2},f]  = [{y_2},[f,h]] = [[{y_2},f],h] - [[{y_2},h],f]  = [a_{{2}f}^0x_0+a_{{2}f}^1 x_1,h] - ma_{{2}f}^1[ x_0, f],$$
and we have
$$\begin{array}{l}
-2a_{{2}f}^0x_0-2a_{{2}f}^1 x_1=ma_{{2}f}^0x_0+a_{{2}f}^1(m-2)x_1-ma_{{2}f}^1x_1 \Rightarrow a_{{2}f}^0=0.
\end{array}$$

Therefore, $[{y_2},f]=a_{{2}f}^1x_1$ and $[{y_2},h]=ma_{{2}f}^1x_0.$

Taking the change ${y_2}'={y_2}-a_{{2}f}^1x_0,$ we obtain
$$\begin{array}{ll}
[{y_2}' ,f] &= [{y_2},f] - a_{{2}f}^1[ x_0, f] = a_{{2}f}^1x_1 - a_{{2}f}^1x_1 = 0,\\[2mm]
[{y_2}' ,h] &= [{y_2},h]- a_{{2}f}^1[ x_0, h] = ma_{{2}f}^1x_0 - ma_{{2}f}^1x_0 = 0.
\end{array}$$

Thus, we have $[R^{-1}, sl_2^{-1}] = 0$ which completes the proof of the theorem.
\end{proof}

\


In the case when the dimension of the ideal $I$ is equal to three, we get the family of Leibniz algebras with the following table of multiplication:
$$\begin{array}{lll} \, [e,h]=2e, & [h,f]=2f, &[e,f]=h, \\[1mm]
\, [h,e]=-2e& [f,h]=-2f, & [f,e]=-h,\\[1mm]
\, [x_1,e]=-2x_0& [x_2,e]=-2x_1, & [x_0,f]=x_1,\\[1mm]
\, [x_1,f]=x_2& [x_0,h]=2x_0, & [x_2,h]=-2x_2,\\[1mm]
\, [e,{y_1}]=\lambda x_0& [f,{y_1}]=\frac{1}{2} \lambda x_2, & [h,{y_1}]=\lambda x_1,\\[1mm]
\, [e,{y_2}]=\mu x_0& [f,{y_2}]=\frac{1}{2} \mu x_2, & [h,{y_2}]=\mu x_1,\\[1mm]
\, [{y_1},{y_2}]={y_1}& [{y_2},{y_1}]=-{y_1}, & [{y_2},{y_2}]=-\frac{ab}{2}x_2,\\[1mm]
\, [x_0,{y_2}]=ax_0& [x_1,{y_2}]=ax_1, & [x_2,{y_2}]=ax_2,\\[1mm]
\, [{y_2},e]=bx_1& [{y_2},h]=bx_2,
 \end{array}$$

Verifying the Leibniz identity to above family of algebras, using a program in the software $Mathematica$ \cite{programLeibniz}, we get the condition $\lambda(1-a)=0.$

Taking the change in the form  ${y_2}'={y_2}+\frac{b}{2}x_2$ we obtain
 $$\begin{array}{ll}
 [{y_2}',e]&=[{y_2}+\frac{b}{2}x_2,e]=[{y_2},e]+\frac{b}{2}[x_2,e]=bx_1-bx_1=0,\\[1mm]
[{y_2}',h]&=[{y_2}+\frac{b}{2}x_2,h]=[{y_2},h]+\frac{b}{2}[x_2,h]=bx_2-bx_2=0,\\[1mm]
[{y_2}',{y_2}']&=[{y_2}+\frac{b}{2}x_2,{y_2}+\frac{b}{2}x_2]=[{y_2},{y_2}]+\frac{b}{2}[x_2,{y_2}]=-\frac{ab}{2}bx_2+\frac{ab}{2}bx_2=0.
\end{array}$$

Thus, we can assume that
$[{y_2}'e]=[{y_2}',h]=[{y_2}',{y_2}']=0$
and we have the family of algebras $L(\lambda,\mu,a).$
$$\begin{array}{lll} \, [e,h]=2e, & [h,f]=2f, &[e,f]=h, \\[1mm]
\, [h,e]=-2e& [f,h]=-2f, & [f,e]=-h,\\[1mm]
\, [x_1,e]=-2x_0& [x_2,e]=-2x_1, & [x_0,f]=x_1,\\[1mm]
\, [x_1,f]=x_2& [x_0,h]=2x_0, & [x_2,h]=-2x_2,\\[1mm]
\, [e,{y_1}]=\lambda x_0& [f,{y_1}]=\frac{1}{2} \lambda x_2, & [h,{y_1}]=\lambda x_1,\\[1mm]
\, [e,{y_2}]=\mu x_0& [f,{y_2}]=\frac{1}{2} \mu x_2, & [h,{y_2}]=\mu x_1,\\[1mm]
\, [x_0,{y_2}]=ax_0& [x_1,{y_2}]=ax_1, & [x_2,{y_2}]=ax_2,\\[1mm]
\, [{y_1},{y_2}]={y_1}& [{y_2},{y_1}]=-{y_1},
 \end{array}$$
with the condition $\lambda(1-a)=0.$

\begin{thm}
Let $L$ be a Leibniz algebra such that $L/I\cong sl_2\oplus R$, where
$R$ is a two-dimensional solvable ideal and $I$ is three-dimensional right irreducible module over $sl_2$. Then algebra
$L$ is isomorphic to one of the following pairwise non isomorphic algebras :
$$L(1, 0, 1);\ L(0, 1, a);\ L(0, 0, a),\mbox{ with } a \in F.$$
\end{thm}

\begin{proof} Similarly as above we derive $\lambda(1-a)=0.$

Let $\lambda\neq0$, then $a=1$. By change
${y_1}'=\frac{1}{\lambda}{y_1}$, ${y_2}'=-\frac{\mu}{\lambda}{y_1}+{y_2}$ we deduce
$$\begin{array}{l}
[e,{y_1}']=[e,\frac{1}{\lambda}{y_1}]=\frac{1}{\lambda}\lambda x_0=x_0,\\[2mm]
[f,{y_1}']=[e,\frac{1}{\lambda}{y_1}]=\frac{1}{2\lambda}\lambda x_2=\frac{1}{2}x_2,\\[2mm]
[h,{y_1}']=[h,\frac{1}{\lambda}{y_1}]=\frac{1}{\lambda}\lambda x_1=x_1,\\[2mm]
[e,{y_2}']=[e,-\frac{\mu}{\lambda}{y_1}+{y_2}]=-\frac{\mu}{\lambda}[e,{y_1}]+[e,{y_2}]=-\frac{\mu}{\lambda}\lambda x_0+\mu x_0=0,\\[2mm]
[f,{y_2}']=[f,-\frac{\mu}{\lambda}{y_1}+{y_2}]=-\frac{\mu}{\lambda}[f,{y_1}]+[f,{y_2}]=-\frac{\mu}{2\lambda}\lambda x_2+\frac{1}{2}\mu x_2=0,\\[2mm]
[h,{y_2}']=[h,-\frac{\mu}{\lambda}{y_1}+{y_2}]=-\frac{\mu}{\lambda}[h,{y_1}]+[h,{y_2}]=-\frac{\mu}{\lambda}\lambda x_1+\mu x_1=0.
\end{array}$$

Thus, we can assume that $\lambda=1$ and $\mu=0$. Hence, we get the algebra
$$L(1,0,1)$$

If $\lambda=0$, then in the case when $\mu\neq0$ by scale of basis of $I$ we can suppose that $\mu=1$, i.e. we obtain the algebra $L(0, 1, a).$

If $\lambda=0$, then in the case when $\mu=0$ we get the algebra
$L(0, 0, a).$

Pairwise non isomorphismness
 of these algebras we obtain by using the program in $Mathematica$ \cite{program}. The theorem is proved.
\end{proof}

\

Analyzing the above obtained results we can formulate

{\bf Conjecture}: Any Leibniz algebra is decomposed into semidirect sum of its corresponding Lie algebra and the ideal I.


\begin{thebibliography}{7}


\bibitem{Abdykasimova} Abdykassymova S., {\it Simple Leibnniz algebras of rank 1 in the characteristic $p$}, Ph.D. thesis, Almaty State University, Kazakhstan, (2001).

\bibitem{AbdDzhum} Abdykassymova S., Dzhumadil'daev A.S. {\it Leibniz algebras in characteristic $p$}, C.R. Acad. Sci. Paris Sr. I Math. 332(12), (2001), 1047--1052.

\bibitem{programLeibniz} Camacho L.M., G\'{o}mez J.R., Gonz\'{a}lez A.J., Omirov B.A., {\it Naturally graded quasifiliform Leibniz algebras,} Journal Symbolic Computation, 44(5), (2009), 527--539.


\bibitem{program} Camacho L.M., Ca\~{n}ete E.M., G\'{o}mez J.R., Omirov B.A., {\it Computational Applications for Maximum Length in Laibniz algebras}, submitted to Advances in Computational Mathematics, (2011).
 \bibitem{Jac}  Jacobson N., Lie algebras, {\itshape Interscience Publishers}, Wiley, New York,
1962.

\bibitem{Kost} Kostrikin A.I., {\itshape Around Bernside} [in Russian],
Nauka, Moskow,  1986.

\bibitem{Lod1} Loday J.-L., Une version non commutative des algebres de Lie: les algebres de Leibniz, {\itshape Enseign. Math.}, {\bf 39},
(1993), 269-293.

\bibitem{Mal1} Mal'cev A., On solvable Lie algebras, (in Russian), {\itshape Bull.Acad.Sci.URSS, Ser.Math. (Izvestia Akad. Nauk SSSR)}, {\bf 9}, (1945),329-356.

\bibitem{Omi} Omirov B.A., On finite dimensional nilpotent Leibniz algebras, Ph. D. Thesis, [in Russian], Tashkent, Uzbekistan, 2001.

\bibitem{Rakh} Rakhimov I.S., Omirov B.A. and Turdibaev R.M., On description of Leibniz algebras corresponding to $sl_2,$ {\itshape } to appear 2012.

\bibitem{Zel} Zelmanov E.I., On the global nilpotency of
Engel Lie algebras in characteristic zero, [in Russian], {\itshape DAN
SSSR}, {\bf4} (1986), 12-26.

\end{thebibliography}
\end{document}